\documentclass[10pt,a4paper]{article}

\usepackage{amsthm}
\usepackage{amsmath}
\usepackage{amsfonts}
\usepackage{amssymb}
\usepackage{graphicx}%

\usepackage[hyperindex = true, colorlinks = true, linktocpage = true, linkcolor=blue, citecolor=blue, bookmarks, pdftitle={Determination of the homotopy classes of maps on surfaces of genus $g$ with infinitely many periodic points}, pdfauthor={Joerg Kampen}]{hyperref}

\usepackage[numbers]{natbib}
\newtheorem{theorem}{Theorem}

\newtheorem{corollary}[theorem]{Corollary}

\newtheorem{definition}[theorem]{Definition}

\newtheorem{lemma}[theorem]{Lemma}

\newtheorem{remark}[theorem]{Remark}

\def\beas{\begin{eqnarray*}}
\def\eeas{\end{eqnarray*}}
\def\bea{\begin{eqnarray}}
\def\eea{\end{eqnarray}}
\def\be{\begin{equation}}
\def\ee{\end{equation}}
\def\bes{\begin{equation*}}
\def\ees{\end{equation*}}
\def\bi{\begin{itemize}}
\def\ei{\end{itemize}}

\renewenvironment{proof}[1][Proof]{\noindent\textbf{#1.} }{\ \rule{0.5em}{0.5em}}
\begin{document}
\title{Determination of the homotopy classes of maps on compact orientable surfaces of positive genus $g$ with infinitely many periodic points (part I) 
}
\author{ J\"org Kampen$^1$}
\maketitle 
\footnotetext[1]{
{\tt kampen@wias-berlin.de}}

\begin{abstract}
We consider the homotopical dynamics on compact orientable surfaces of positive genus $g$. We establish a sufficient and necessary algebraic criterion for homotopy classes  with infinitely many periodic points of  maps on such surfaces in terms of the characteristic polynomial of the matrix representing the correspondig homomorphism of the first homology group. Our methods differ from that of the literature, which typically uses Nielsen number theory, because Nielsen numbers of iterates are difficult to compute except in special cases where they can be easily defined in terms of the easily computable Lefschetz number. Instead we use some fixed point theory from the point of view of real algebraic geometry, and some intersection theory in the context of De Rham theory on surfaces and combine it with some observations of  dynamical system theory which started from the work of Sarkovskii. However, in the special case of the homotopical dynamics of a two dimensional torus, where the Nielsen number is determined by the modulus of the Lefschetz number, we recover a result in \cite{ABLSS}.

\\[0.6ex] {\it 2000 AMS subject
classification:}  37C25
\end{abstract}

\section{Introduction}

Consider a map $f:X\rightarrow X$ on a triangulable compact surface $X$. In general it is well-known that a complete set of homeomorphism classes of surfaces is given by the set of spheres with $g$ handles and spheres with $k$ crosscaps where $g$ and $k$ are arbitrary nonnegative integers. The former class of spheres with $g$ handles are singled out by orientability (cf. \cite{SeT}). In the case $g=0$ any compact orientable surface is homeomorphic to the ordinary sphere $S^2$. However, this case is special because the sphere is simply connected. Especially, its first order homology group is zero, and for this reason the method proposed here cannot be applied directly in this case. Hence, we assume $g\geq 1$ in the following. In order to investigate fixed points of iterated maps on the surface $X$ we shall consider realizations of $X$ by spheres with $g$ handles or by $4g$-sided polygons in normal form with the usual identification as a starting point of our considerations of the discrete dymamics without loss of generality , i.e. we consider iterates $f^n,~n\geq 1$ of $f$ on such realizations. For an integer $k$ we say that $f$ has a periodic point of order $k$ if 
\begin{equation}\label{per1}
f^k(x)=x\mbox{ and }f^i(x)\neq x \mbox{ for }~1\leq i\leq k-1.
\end{equation}
If  the set
\begin{equation}
\mbox{Per}(f):=\left\lbrace k|f~\mbox{has a periodic point of order $k$ }\right\rbrace 
\end{equation}
is of infinite cardinality, then we say that $f$ has infinitely many periodic points. In this paper we aim at a sufficient and necessary condition for the infinite cardinality of the set  minimal periods attained within a homotopy class of a given map $f:X\rightarrow X$, where $X$ is an orientable triangulable compact surface of genus $g\geq 1$ (or some homeomorphic realization such as a sphere with $g\geq 1$ handles $S_g$, or a $4g$-sided polygon in normal form with the usual identifications of sides). If a map $h:X\rightarrow X$ is homotop to a map $f:X\rightarrow X$, then we write $h\simeq f$, and we define the minimal set of periods of $f$ to be the set
\begin{equation}
\mbox{MPer}\left(f\right)=\cap_{h\simeq f} \mbox{Per}\left(h\right). 
\end{equation}
If $\mbox{MPer}\left(f\right)$ is of infinite cardinality, then we say that the homotopy class of $f$ has infinitely many periodic points.
The standard method for investigation of $\mbox{MPer}\left(f\right)$ is via Nielsen number theory. The Nielsen number $N(f)$ of a map $f$ defined on a compact Euclidean neighborhood retracts is a homotopy invariant which gives a lower bound for the number of fixed points of $f$ (for a definition compact Euclidean neighborhood retracts cf. \cite{D}). In the special case of an $n$-torus the Nielsen number can be defined to be the modulus of the Lefschetz number by a result in \cite{BBPT}, i.e. for a map $f:{\mathbb T}^n\rightarrow {\mathbb T}^n$ we may define
\begin{equation}\label{NL}
N(f)=|L(f)|,
\end{equation}
where
\begin{equation}
L(f)=\sum_{k=0}^ m(-1)^k\mbox{Trace}\left(f_{*k}\right), 
\end{equation}
and where $f_{*k}:H_k\left(X,{\mathbb Z} \right)\rightarrow H_k\left(X,{\mathbb Z} \right)  $ is the $k$th order homology morphism of $f$. This relation is important, since the Lefschetz number is easy to compute while the Nielsen number is difficult to compute in general. This is also the essential result used in \cite{ABLSS} in order to give a full characterization of minimal periods of maps on the $2$-dimensional torus. However, (at least to the extent of my knowledge) next to the $n$-torus there are only very special manifolds where the relation (\ref{NL}) holds. Especially compact orientable surfaces of genus $g$ are not known to be among this class (in the literature the name '$n$-fold torus' for a sphere with $g$ handles may lead to confusion here; the $n$-torus is the higher dimensional object which may be obtained by identification of opposite sides of a hypercube).   
This motivates a theory which avoids the Nielsen number theory and goes back to some very elementary considerations concerning fixed point theory from the point of view of real algebraic geometry. Other ingredients are also very elementary. They consist of elementary intersection theory and observation in dynamical system theory. Especially Sarkovskii observed that period three of a continuous map $f:\left[0,1\right]\rightarrow  \left[0,1\right]$ implies all other periods. The algebraic condition which we establish is defined in terms of the first homology group morphism associated with the map $f:X\rightarrow X$.  More precisely, let
\begin{equation}
f_{*1}:H_1\left(X,{\mathbb Z} \right)\rightarrow H_1\left(X,{\mathbb Z} \right) 
\end{equation}
denote the induced first homology map, corresponding to some $2g\times 2g$ matrix of integers if $X$ is an orientable compact surface of positive genus $g$. In the following we shall write $H_1X$ instead of $H_1\left(X,{\mathbb Z} \right)$ for brevity of notation. Recall that the first homology functor is also homotopy invariant.
In order to prove that a map $f:X\rightarrow X$ has infinitely many periodic points it suffices to find for each prime number $p$ a fixed point free area $S_p\subset X$ such that
\begin{equation}
 f^p(x)=x \mbox{ and  }x\in S_p,
\end{equation}
for infinitely many $p\in \mbox{Prim}$  where $\mbox{Prim}$ denotes the set of prime numbers. For if $f^i(x)=x$ for some $1\leq i\leq p-1$, where $p$ is a prime number, then we have $\gcd(i,p)=1$, and this this would imply that $il+pk=1$ for some integers $k,l\in{\mathbb Z}$, and where one integer is positive and the other negative. If $l<0$ then we have
\begin{equation}
x=f^{pk}=f^{(-l)i+1}(x)=f\left( f^{(-l)i}(x)\right) =f(x),
\end{equation}
in contradiction to the assumption that $x\in S_p$ and $S_p$ is assumed to be fixed-point free. Similarly, if $k<0$ then we have 
\begin{equation}
x=f^{li}=f^{(-k)p+1}(x)=f\left( f^{(-k)p}(x)\right) =f(x),
\end{equation}
again in contradiction to the assumption that $x\in S_p$, because $S_p$ is assumed to be fixed-point free. It is well known that the first homology group of a sphere with $g$ handles is a free abelian group of rank $2g$, and we denote the matrix corresponding to the first homology homomorphism by $M_{2g}:{\mathbb Z}^{2g}\rightarrow {\mathbb Z}^{2g}$. As the matrix $M_{2g}$ encodes some information of the map $f:X\rightarrow X$, iterates of the matrix $M_{2g}$ encode some information of the iterates $f^m :X\rightarrow X$. If the matrix $M_{2g}$ is diagonalizable with some matrix $S$, i.e. $S^{-1}AS=\Lambda$ for some diagonal matrix $\Lambda$, then parallel to topological conjugacy we have
\begin{equation}
M_{2g}^n=S^{-1}\Lambda^n S.
\end{equation}
Hence in this case we see easily that essential information of the iterates of the first homology homomorphism is encoded in the characteristic polynomial $p$ of the matrix $M_{2g}$. Let $p$ denote the characteristic polynomial of the matrix $M_{2g}$. In the general case we have a similar relation with a Jordan normal form over the complex numbers ${\mathbb C}$ where iterates are iterates of Jordan blocks of the Jordan matrix. In this paper we shall derive algebraic conditions in terms of the matrix $M_{2g}$ and its Jordan normal form over the comples numbers ${\mathbb C}$ (resp. the zeros of its characteristic polynom) which are necessary and sufficient for a map on compact orientable surfaces to have homotopy classes with infinitely many periodic points. Note that analytically a full characterization of these homotopy classes may be possible only for the $2$-torus and the 'Brezel'-surface, where $g=2$ (Abel's no go). However, even in the general case we are able to get a decidable (even efficiently decidable) algebraic condition. In this first part of the paper we state and prove an algebraic condition for homotopy classes with infinitely many periodic points.

The outline of the paper is as follows. In the next section we state the main theorem which provides a necessary and sufficient algebraic condition in order to decide whether all the maps of a homotopy class on a given compact orientable surface $X$ of positive genus $g\geq 1$ have infinitely many periodic points. We also prove a corollary which covers a result for compact orientable surfaces of genus 1, i.e. surfaces homeomorphic to the $2$-torus. This corollary is also proved in the literature by the use of Nielsen number theory. Then in section 3 we consider some fixed point theory from the point of view of real algebraic geometry. We pove an extension of Brouwer's fixed point theory and show that fixed points for a dense class of mappings of polynomial type can be constructed as intersections of algebraic curves. It is easily observed that a fixed point property is preserved by a continuous mapping constructed as a limit of certain mappings of polynomial type. Finally in section 4 we prove the main theorem using fixed point theory from the point of view of real algebraic geometry. We extend our considerations of section 3 on fixed points as intersection of algebraic curves. If the algebraic condition of the main theorem is satisfied, we can identify fixed-point free semi-algebraic sets which contain fixed points of iterates which are intersections of algebraic curves which are homoptop to loops corresponding to basic loops of the homotopy group of the surface $X$. Elementary intersection theory in the framework of De Rham cohomology tells us that the intersections are not empty. 

\section{Statement of main theorem}
Given a compact orientable surface $X$ of positive genus $g$ we know that it is homeomorphic to a sphere with $g$ handles and can be realized via a $4g$-sided polygon $\Pi$ with the usual identifications. Given $2g$ symbols $\alpha_i,~1\leq i\leq g$ and $\beta_i,~1\leq i\leq g$, assume that the normal form of the $4g$-sided polygon $P_{4g}$ realizing $X$ is given by the code
\begin{equation}\label{pathc}
\alpha_1\cdot\beta_1\cdot\alpha_1^{-1}\cdot\beta_1^{-1}\cdot\alpha_2\cdot\beta_2\cdot\alpha_2^{-1}\cdot\beta_2^{-1}\cdot \cdots\cdot \alpha_g\cdot \beta_g\cdot\alpha_g^{-1}\cdot \beta_g^{-1},
\end{equation}
where we may think of the symbols $\alpha_i,\beta_i,~1\leq i\leq g$ to be the sides of a polygon (corresponding to certain loops of the surface $X$). We may also think of the symbols $\alpha_i,~1\leq i\leq g$ and $\beta_i,~1\leq i\leq g$ as  elements of equivalence classes of the fundamental group $\pi_1\left( X,x\right) $ for some $x\in X$ which may correspond to a kink of the $4g$-sided polygon. Then we may interprete (\ref{pathc}) as a product path $\gamma$ (the path around the sides of the convex polygon $\Pi$). The path corresponds to an element in a free abelian group of $2g$ generators $F_{2g}$ and its least normal subgroup $N_{2g}$ leads us to the isomorphism
\begin{equation}
\pi_1\left(X,x\right)=F_{2g}/N_{2g}, 
\end{equation}
by use of van Kampen's theorem. It follows that $H_1X$ is a free abelian group of rank $2g$ with basis the image of the loops $\alpha_1,\beta_1,\cdots ,\alpha_g,\beta_g$ (well, this basis is designed for the fundamental group where a base point has to be prescribed, later we shall use the natural basis of the first homology group). Next, assume that the
 $2g\times 2g$-matrix $M_{2g}$ represents the first order homology homomorphism with respect to the order of the normal form.
The matrix $M_{2g}$ is similar to its Jordan normal form (over the complex numbers ${\mathbb C}$)
\begin{equation}\label{Jordan} J= \begin{pmatrix} J_1 & & 0 \\ & \ddots & \\ 0 & & J_k \end{pmatrix} = S^{-1}M_{2g}S,
\end{equation}
where $S$ is an invertible matrix and for $1\leq i\leq k$ the Jordan blocks of dimension $n_i$ are given in the form
\begin{equation}
     J_i= \begin{pmatrix} \lambda_i & 1 & & & 0 \\ & \lambda_i & 1 & & \\ && \ddots{} & \ddots{}\\ &&& \lambda_i & 1 \\ 0 & & & & \lambda_i \end{pmatrix} ,
\end{equation}
and with $\sum_{i=1}^k n_i=2g$. We denote by $d_j$ the $j$th diagonal element of $J$. Note that for any $j$ there is a unique number $l$ with $0\leq l\leq k-1$ such that $j=\sum_{i=1}^{l}n_i+r$ with $1\leq r\leq n_{l+1}$ (we denote $\sum_{i=1}^{0}n_i=0$). Then $d_j=\lambda_{l+1}$. Each pair of diagonal elements $(d_{2i-1},d_{2i}),~1\leq i\leq g$ corresponds to the pair  $(a_i,b_i)$, where $a_i$ and $b_i$ belong to the standard basis of the first homology group of $X$. For a given $j$ with $1\leq j\leq g$, we say that a pair $(d_{2j-1},d_{2j})$ is hyperbolic if either
\begin{equation}\label{condition2}
|d_{2j-1}|< 1,|d_{2j}|>1,
\end{equation}
or
\begin{equation}\label{condition3}
|d_{2j-1}|> 1,|d_{2j}|<1
\end{equation}
is satisfied. Furthermore if for a given $j$ with $1\leq j\leq g$ we have
\begin{equation}\label{condition1}
|d_{2j-1}|> 1,|d_{2j}|>1,
\end{equation}
then we say that a pair $(d_{2j-1},d_{2j})$ is expansive.
\begin{theorem}
Let $f:X\rightarrow X$ be a compact orientable surface of genus $g$ with $g\geq 1$, and assume that the induced homology homomorphism
\begin{equation}
\left(f_X\right)_{*1}:H_1X\rightarrow H_1X 
\end{equation}
has a matrix representative $M_{2g}$.
Let $J=S^{-1}M_{2g}S$ be the Jordan normal form over ${\mathbb C}$ of the matrix $M_{2g}$. If there is an index $1\leq i\leq g$ and a pair of diagonal elements $(d_{2i-1},d_{2i})$ of $J$ which are either hyperbolic (i.e. either condition (\ref{condition2}) or condition (\ref{condition3}) is satsified) or expansive (i.e. (\ref{condition1}) is satisfied), then 
\begin{equation}
\mbox{card}\left( \mbox{MPer}(f)\right)=\infty, 
\end{equation}
where $\mbox{card}(M)$ denotes the cardinality of a set $M$. Note that conditions (\ref{condition2}) and (\ref{condition3}) are only realized if $d_i$ and $d_{i+1}$ belong to different Jordan blocks.
\end{theorem}
Before we go into the proof let us first consider the case of a $2$-torus ($g$=1). Here the question which homotopy classes have infinitely many periodic points is answered completely in \cite{ABLSS}. Let us look if the algebraic criterion given above leads to the same result. Indeed, we have
\begin{corollary}
Let $f:{\mathbb T}^2\rightarrow {\mathbb T}^2$ be a map of the torus, and let $M_2\in {\cal M}_{2\times 2}\left({\mathbb Z}\right)$ be the matrix corresponding to the first homology morphism $f_{*1}$. Furthermore, assume that the $t=t(M_2)$ and $d=d(M_2)$ denote the trace and the determinant of the matrix $M_2$. If the pair $(t,d)\in {\mathbb Z}^2$ is {\sc not} on the line
\begin{equation}
-t+d+1=0
\end{equation}
and {\sc not} in the set
\begin{equation}
M_6=\left\lbrace (0,0),~(-1,0),~(-2,1),~(0,1),~(-1,1),~(1,1)\right\rbrace ,
\end{equation}
then 
\begin{equation}
\mbox{card}\left( \mbox{MPer}(f)\right)=\infty. 
\end{equation}
\end{corollary}
\begin{proof}
We verify the criterion of theorem 1 above. $M_2$ has the eigenvalues
\begin{equation}
\lambda_1=\frac{t-\sqrt{t^2-4d}}{2},~~\lambda_2=\frac{t+\sqrt{t^2-4d}}{2}
\end{equation}

First we list the pairs of eigenvalues for all pairs $(t,d)\in M_6$. We have
\begin{equation}\label{sixcases}
\begin{array}{ll}
(t,d)=(0,0)\rightarrow (\lambda_1,\lambda_2)=(0,0)\\
\\
(t,d)=(-1,0)\rightarrow (\lambda_1,\lambda_2)=(-1,0)\\
\\
(t,d)=(-2,1)\rightarrow (\lambda_1,\lambda_2)=(-1,-1)\\
\\
(t,d)=(0,1)\rightarrow (\lambda_1,\lambda_2)=(-i,i)\\
\\
(t,d)=(-1,1)\rightarrow (\lambda_1,\lambda_2)=\left( \frac{-1-i\sqrt{3}}{2},\frac{-1+i\sqrt{3}}{2}\right) \\
\\
(t,d)=(1,1)\rightarrow (\lambda_1,\lambda_2)=\left( \frac{1-i\sqrt{3}}{2},\frac{1+i\sqrt{3}}{2}\right).
\end{array}
\end{equation}
Note that except in the case $t=d=0$ in each of the six cases we have at least one eigenvalue of modulus $1$. Hence, the condition of theorem 1 for infinitely many periodic points is violated in each of the six cases (\ref{sixcases}). Furthermore, if $(t,d)$ satisfies $-t+d+1=0$, then we have
\begin{equation}
(\lambda_1,\lambda_2)=(1,d),
\end{equation}
hence the condition of theorem 1 for infinitely many periodic points is violated again. In all other cases we find that the condition of theorem 1 is satisfied. 
\end{proof}

\section{Some fixed point theory from the point of view of real algebraic geometry}

From the point of view of real algebraic geometry Brouwers fixed point theorem can be proved by some basic observations about intersections of algebraic varieties. This point of view leads immediately to an extension of the Brouwerian fixed point theorem which were proved with different methods in \cite{K}. In this paper we only consider the dimension $n=2$. Indeed, in order to prove the existence of a fixed point of a continuous map $f:Q\rightarrow {\mathbb R}^2$, where $Q$ is a square, $Q=\left[0,1\right]^2$, it suffices to consider a sequence of polynomial mappings with restriction to $Q$, i.e. polynomial mappings ${\mathbf p}^n:Q\rightarrow {\mathbb R}^2$, where each ${\mathbf p}^n$ is a pair ${\mathbf p}^n=(p^n_1,p^n_2)$ of polynoms restricted to $Q$ with bivariate polynoms $p^n_i$ (here and in the following we identify bivariate polynoms in the ring ${\mathbb R}\left[X_1,X_2\right]$  with certain maps in the obvious way if this is convenient). For $i\in \left\lbrace 1,2\right\rbrace $ we say that $f=(f_1,f_2)$ is expansive in direction $i$ 
\begin{equation}
f_i(C)\supseteq \left[0,1\right]  
\end{equation}
for any algebraic curve $C\subset Q$ connecting the faces $\left\lbrace 0\right\rbrace\times \left[0,1\right]$ and $\left\lbrace 1\right\rbrace\times \left[0,1\right]$, i.e. $C$ consists of one pathwise connected component and 
\begin{equation}
C\cap \left\lbrace 0\right\rbrace\times \left[0,1\right]\neq \oslash,
\end{equation}
 and
\begin{equation}
C\cap \left\lbrace 1\right\rbrace\times \left[0,1\right]\neq \oslash.
\end{equation}
Furthermore, for $i\in \left\lbrace 1,2\right\rbrace $ we say that $f=(f_1,f_2)$ is conservative in direction $i$ if
\begin{equation}
f_i(Q)\subseteq \left[0,1\right].
 \end{equation}
 We have
 
 \begin{theorem}
 Let $f:Q\rightarrow {\mathbb R}^2$ be a continuous map which is either expansive or conservative in all directions $i\in \left\lbrace 1,2\right\rbrace$. Then $f$ has a fixed points, i.e. we have $f(x)=x$ for some point $x\in Q$.  
 Moreover for a dense class of polynomial mappings at least one fixed point is the intersection of two algebraic curves $C_1,C_2$ on $Q$, where $Q\setminus C_i,~1\leq i\leq 2$ consists of at least two connected components. Moreover, there is a dense set of polynomial mappings with an odd number of fixed points.
 \end{theorem}
 The latter theorem is an extension of the Brouwerian fixed point theory by topological conjugacy, i.e. we may assume (by topological conjugacy) that $K$  is the unit square $Q$, where $Q=[0,1]\times [0,1]$. We have
 \begin{corollary}
  Let $g : K\rightarrow K$ be a continuous function where $K\subset {\mathbb R}^2$ is homeomorphic to a ball in ${\mathbb R}^2$. Then $f$ has a fixed point. Moreover, there is a dense class of polynomial mappings in $C(K,K)$ (with respect to the standard topology) with an odd number of fixed points in $K$. 
 \end{corollary}

\begin{proof}
 Let $f=(f_1,f_2):Q\rightarrow Q$ be a continuous  map. Each function $f_i:Q\rightarrow [0,1]$ has a polynomial approximation $p^{n,\epsilon}_i:Q\rightarrow [0,1]$ such that
\begin{equation}
p^{n,\epsilon}_i(x,0)\neq 0,~~p^{n,\epsilon}_i(x,1)\neq 1~~\mbox{for}~~x\in [0,1],
\end{equation}
and
\begin{equation}
\sup_{z\in Q}|f(z)-p^{n,\epsilon}_i(z)|\leq \epsilon .
\end{equation}
Next w.l.o.g. we may assume that the univariate polynomial $p^{n,\epsilon}_1(0,Y)$ has only simple zeros. Clearly the parity of the number of simple zeros is odd. Now the equation $p^{n,\epsilon}_1(X,Y)-X=0$ restricted to $Q$ defines an algebraic curve $C_1$ on $Q$. Since the number of half-branches at each point $z\in C_1$ is even, and the intersection with the sets $\left\lbrace z=(z_1,z_2)\in Q|z_1=0\right\rbrace$, and $\left\lbrace z=(z_1,z_2)\in Q|z_1=1\right\rbrace$ is the empty set, there is at least one connected component $K\subseteq C_1$ which meets opposite faces of the square, i.e. $K_1\cap \left\lbrace z=(z_1,z_2)\in Q|z_2=0\right\rbrace\neq \oslash$ and $K_1\cap \left\lbrace z=(z_1,z_2)\in Q|z_2=1\right\rbrace\neq \oslash$.  Similarly, the equation $p^{n,\epsilon}_2(X,Y)-Y=0$ restricted to $Q$ defines an algebraic curve $C_2$ on $Q$. Again, there is at least one connected component $K_2\subseteq C_2$ which meets opposite faces of the square, i.e. $K_2\cap \left\lbrace z=(z_1,z_2)\in Q|z_1=0\right\rbrace\neq \oslash$ and $K_2\cap \left\lbrace z=(z_1,z_2)\in Q|z_1=1\right\rbrace\neq \oslash$. Note that every point of $K_1\cap K_2$ is a fixed point. 
Next we may use some real algebraic geometry.
\begin{lemma}
Let $C\subset {\mathbb R}^n$ be a real algebraic curve and let $x$ be a point with $z\in C$. Then the number of half-branches of $C$ centered at $z$ is even. 
\end{lemma}

\begin{proof} If $x$ is nonsingular, then the number of half-branches of $x$ is equal to two (because then in a neighborhood of $z$ $C$ is Nash diffeomorphic to an open interval of ${\mathbb R}$. On the other hand if $z=(x,y)$ is a singular point of $C$, then let $q(X,Y)=0$ be an equation of $C$ without multiple factor. W.l.o.g. we may assume that for each $a\in [0,1]$ on the line $X=a$ $q$ is monic with respect to $Y$. Furthermore, we may assume w.l.o.g. that $z$ is the only singular point on the line $X=z_1$, where $z=(z_1,z_2)$ (otherwise change coordinates appropriately). Next we may choose $\epsilon >0$ such that the discriminant of $q$ with respect to $Y$ has no zero in the open intervals $(z_1-\epsilon , z_1)$ and $(z_1 , z_1+\epsilon)$. The real roots of $q(x,Y)$ are given by the Nash functions $\alpha_1<\cdots <\alpha_k$ for the Nash functions $\alpha_i,~1\leq i\leq k$ for $x\in (z_1-\epsilon , z_1)$ and by $\beta_i,~1\leq i\leq l$ for  $x\in (z_1 , z_1+\epsilon)$, with limit $z$ at $z_1$. Since the number $k$ and $l$ have the same parity as $q$ with respect to $Y$, we conclude that $k+l$ is even.
\end{proof}

 If the sequence ${\mathbf p}^n$ converges to the map $f$ with respect to the supremum norm (well consider convolution with the heat kernel and cut off or consider some related construction in the context of the Weierstrass approximation theorem), then it suffices to prove existence of fixed points for any polynomial mapping ${\mathbf p}^n$. Indeed, a proof leads to a sequence of fixed points $x_n$ with ${\mathbf p}^n(x_n)=x_n$, and since $(x_n)\subset Q$ with $Q$ compact there is a subsequence $(x_m)$ which has a limit $x$. Since
\begin{equation}\label{fixed}
f(x)-x\leq |f(x)-{\mathbf p}^n(x)|+|{\mathbf p}^n(x)-{\mathbf p}^n(x_m)|+|x_m-x|,
\end{equation}
we get the fixed point $x$ for $f$ as we take first the limit with respect to $m$ and then with respect to $n$ on the right side of (\ref{fixed}).
\end{proof}

\section{Proof of theorem 1}

We shall prove the theorem first for local polynomial mapping approximations and then show that the properties claimed are preserved in the limit. 
More precisely, for each prime number $p$ we shall construct a local polynomial mapping approximation (of a suitable topological conjugate) of $f^p$, i.e. of the $p$th iteration of the map $f$. If the number $p$ is large enough, then we identify a fixed point free set $S_p\subset X$, i.e. $f(x)\neq x$ for any $x\in S_p$ where a fixed point of the approximation of the map $f^p$ is constructed as an intersection of two real algebraic varieties. Here we use that the algebraic varieties contain images of loops defining a first order homology class  in $H_1X$ and use elementary intersection theory:
given a $C^{\infty}$ model of $X$ and any $\gamma\in H_1X$ recall that there is a unique class $\omega_{\gamma}\in H^1X$ in the cohomology class $H^1X$ such that
\begin{equation}
\int_{\gamma}\delta =\int_X \omega_{\gamma}\wedge \delta 
\end{equation}
for all $\delta\in H^1X$. Then for any $\sigma,\tau \in H_1X$ define the intersection number $\left\langle \sigma,\tau\right\rangle $ to be
\begin{equation}
\left\langle \sigma,\tau\right\rangle =\int_X\omega_{\sigma}\wedge \omega_{\tau}.
\end{equation}
This number is an integer, but what we really need is the well-known fact that with the basis $(a_i,b_i),~1\leq i\leq g$ for $H_1X$ we have
\begin{equation}
\left\langle a_i,b_j\right\rangle :=\left\lbrace 
\begin{array}{ll}
1,~~\mbox{if}~i=j\\
\\
0,~~\mbox{if}~i\neq j .
\end{array}\right.
\end{equation}
It is well-known that the intersection number is a homotopy invariant which makes it useful in our conext. 

Next we consider polynomial mappings which have fixed points which are intersections of algebraic curves with intersection number $1$. First let us identify the pair $\left( \lambda_i,\lambda_{i+1}\right)$ satisfying the assumptions of theorem $1$. The members $\lambda_i,\lambda_{i+1}$ of the pair belong either to the same Jordan block $J_l$ for some positive integer $l$ or to different Jordan blocks $J_l$ and $J_{l+1}$. If they belong to the same Jordan block we have $\lambda_i=\lambda_{i+1}$ for some positive integer $i$. In that case the assumptions of theorem 1 require that a) $|\lambda_i|=|\lambda_{i+1}|>1$. If they are in different Jordan blocks, then the assumptions of theorem 1 require that one of the three situations b) $|\lambda_i|,|\lambda_{i+1}|>1$, or c) $|\lambda_i|<1~\&~|\lambda_{i+1}|>1$, or d) $|\lambda_i|>1~\&~|\lambda_{i+1}|<1$. In the cases a) and b) we say that the map $f$ is expansive with respect to the basis elements $a_m,b_m$ of the first homology group $H_1X$, where $i=2m+1$ along with some $m\in \left\lbrace 0,\cdots ,g-1\right\rbrace $. In the cases c) and d) we say that the map $f$ is hyperbolic with respect to the basis elements  $a_m,b_m$. Assume that a) is true, i.e. that $|\lambda_i|=|\lambda_{i+1}|>1$. Then $\lambda_i,\lambda_{i+1}$ with $\lambda_i=\lambda_{i+1}$ are elements of the diagonal of one Jordan block $J_r$ for some $1\leq r\leq k$ (if there are $k$ Jordan blocks in the normal form over ${\mathbb C}$).
From (\ref{Jordan}) we get
\begin{equation}\label{iterate}  S J^l  S^{-1}= S\begin{pmatrix} J_1^l & & 0 \\ & \ddots & \\ 0 & & J_k^l \end{pmatrix}^l  S^{-1}= M_{2g}^l,
\end{equation}
W.l.o.g. we may assume that $i=1$ and $i+1=2$, with the corresponding images of the loops $a_1,b_1$ of the basis of $H_1X$. Jordan normal form theory tells us that we may assume that the first two diagonal elements of $S^{-1}=\left(s^{ij}\right) $ are not zero, i.e. $s^{11}\neq 0$ and $s^{22}\neq 0$. Then for the unit vectors $e_1=(1,0,\cdots ,0)$ $e_2=(0,1,\cdots ,0)$ we get
\begin{equation}\label{iterMg1}
|P_1M_{2g}^me_1|\uparrow \infty~\mbox{as}~m\uparrow \infty
\end{equation} 
and 
\begin{equation}\label{iterMg2}
|P_2M_{2g}^me_2|\uparrow \infty~\mbox{as}~m\uparrow \infty ,
\end{equation}
where $P_1$ (resp. $P_2$) denotes the projection on the first (resp. second) component. Here $e_1$ and $e_2$ correspond to the loops $a_1,b_1$ or to the images of the loops which belong to the basis which generates the first homology group $H_1X$. In case b) we have an analogous situation where (\ref{iterMg2}) holds. In cases c) and d) we have that
\begin{equation}\label{iterMgcd1}
|P_1M_{2g}^me_1|\uparrow \infty~\mbox{as}~m\uparrow \infty,
\end{equation} 
or
\begin{equation}\label{iterMgcd2}
|P_2M_{2g}^me_2|\uparrow \infty~\mbox{as}~m\uparrow \infty 
\end{equation}
holds, but not both in general.
Next consider a homeomorphism $\phi :~{\mathbb R}^2\rightarrow {\mathbb R}^2$ such that
\begin{equation}
 \phi\left(\Pi_{4g}\right)=Q\cup \Pi^*_{4g-4},
\end{equation}
where $Q=[0,1]^2$, and $\Pi^*_{4g-4}\subset {\mathbb R}^2$ is a $4g-4$ sided polygonal such that 
\begin{equation}
 Q\cap \Pi^*_{4g-4} =\left\lbrace 1\right\rbrace \times \left[ \frac{1}{2},1 
  \right] 
\end{equation}
Note that a surface homeomorph to $X$ can be obtained by identifcations of sides of  $\phi\left(\Pi_{4g}\right)$ according to the rules induced by the identifcations of $\Pi_{4g}$ via the homeomorphism $\phi$. Especially, the sides $\left\lbrace 0\right\rbrace \times [0,1] $ and $\left[ \frac{1}{2},1 
  \right]$ are identified in such a procedure.

Now consider a family of paths $c_{x_0}^{i,x_0+(0,1)}:[0,1]\rightarrow Q,ĩ\in I$ for some index set $I$ which connect $x_0=(x_{01},x_{02})$ with $x_0+(0,1)=(x_{01},x_{02}+1)$ in $Q$ and a family of paths $c_{(1,0)}^{i,(1,0.5)}:[0,1]\rightarrow Q,ĩ\in I$ for some index set $J$ which connect $x_0=(0,1)$ with $(1,0.5)$ in $Q$. 

These paths correspond to loops on the surface $X$ by the indicated identification of paths, say to loops homotop to $a_1$ w.l.o.g.. Now let ${\mathbf f}^p_{\tiny \mbox{loc}}=\left(f^p_{\tiny \mbox{loc},1},f^p_{\tiny \mbox{loc},2} \right):A_p\rightarrow Q$ be a local representation of the map ${\mathbf f}^p$ on the polygon ($A_p$ maximal).

As the degree of maps $g^p$ defined on images of loops homotop to $a_1$ and induced by ${\mathbf f}^p$ increases with the prime number $p$ analog as in the fixed point theory above we find an algebraic curve $\Gamma_{1p} 
\subseteq \left( f^p_{\tiny \mbox{loc},2} \right)^{-1}(1)$ (a subset of $A_p$) which connects $\left\lbrace 1\right\rbrace \times \left[ 0,\frac{1}{2} 
  \right]$ and $\left\lbrace 0\right\rbrace \times \left[ 0,1 
  \right]$. For $p$ large enough we also find an algebraic curve  
$\Gamma_{2p}\subseteq \left( f^p_{\tiny \mbox{loc},2} \right)^{-1}(0)$ in $A_p$ which connects $\left\lbrace 1\right\rbrace \times \left[ 0,\frac{1}{2} 
  \right]$ and $\left\lbrace 0\right\rbrace \times \left[ 0,1 
  \right]$. Since $f_{\tiny \mbox{loc},2}$ is a polynomial mapping we know that $f^p_{\tiny \mbox{loc},2}$ are polynomial mappings. Furthermore note that these curves are constructed via Nash functions at branching points analogous to the construction above of fixed points as intersections of real algebraic curves above. Removing brances if necessary we may assume that both algebraic curves $\Gamma_{0p}$ and $\Gamma_{1p}$ are homeomorph to the interval. Furthermore we assume that both have exactly one point of intersection with the two opposite faces of a cube connected by the curves, i.e. we assume that there exists exactly one point $\gamma_{00p}$ with
  \begin{equation}
  \gamma_{00p}\in \Gamma_{0p}\cap \left\lbrace 0\right\rbrace \times \left[0,1 \right], 
  \end{equation}
and one point $\gamma_{01p}$ with
\begin{equation}
  \gamma_{01p}\in \Gamma_{1p}\cap \left\lbrace 0\right\rbrace \times \left[0,1 \right], 
  \end{equation}
and exactly one point $\gamma_{11p}$ with
  \begin{equation}
  \gamma_{11p}\in \Gamma_{1p}\cap \left\lbrace 1\right\rbrace \times \left[0,1 \right], 
  \end{equation}
and one point $\gamma_{12p}$ with
\begin{equation}
  \gamma_{12p}\in \Gamma_{1p}\cap \left\lbrace 1\right\rbrace \times \left[0,1 \right], 
  \end{equation}
Next consider the set 
 \begin{equation}
 C_{\Gamma}:= \Gamma_{0p}\cup \left\lbrace 0\right\rbrace \times \left[ \gamma_{00p},\gamma_{01p}\right] \cup \Gamma_{1p} \cup  \left\lbrace 1\right\rbrace \times \left[ \gamma_{11p},\gamma_{12p}\right]
\end{equation}
This set is homeomorphic to a circle and according to the Jordan curve theorem it determines different connected components of $Q$. We may assume that 
\begin{equation}
\left[0,1\right]\times \left\lbrace 0\right\rbrace \cap \left( \Gamma_{0p}\cup \Gamma_{1p}\right) =\oslash ,
\end{equation}
and that
\begin{equation}
\left[0,1\right]\times \left\lbrace 1\right\rbrace \cap \left( \Gamma_{0p}\cup \Gamma_{1p}\right) =\oslash .
\end{equation}
Well, according to our construction the set $Q\setminus C_{\gamma}$ consists of three connected components. Let $x_0\in Q\setminus C_{\gamma}$ be a point which is not in the same connected component as points in the set $\left[0,1\right]\times \left\lbrace 0\right\rbrace\cup \left[0,1\right]\times \left\lbrace 1\right\rbrace$. Let $S_p$ be the closure of the connected which contains the point $x_0$. From our construction it is clear that $S$ is a semi-algebraic set, where we recall
\begin{definition}
A semi-algebraic subset of ${\mathbb R}^n$ is a subset of the form
\begin{equation}
 \bigcup_{l=1}^r\bigcap_{k=1}^s \left\lbrace x\in {\mathbb R}^n|p_{k,l}\ast_{k,l}0\right\rbrace ,
\end{equation}
where $p_{k,l}\in {\mathbb R}\left[X_1,\cdots ,X_n\right] $ and $\ast_{k,l}\in \left\lbrace <,=\right\rbrace $ for $1\leq k\leq r$ and $1\leq l\leq s$.
\end{definition}
Next we observe that the set $S_p$ constructed above is what we call an expansive (case a) and b)) or hyperbolic (case c) and d)) semi-algebraic set. First note that the set $S_p$ as constructed above is a subset of a semi-algebraic set of form
\begin{equation}
A_p:=Q\setminus \left\lbrace p_0>0\right\rbrace \subset Q\setminus \left\lbrace x|\mbox{dist}_E\left(x,\left\lbrace 1 \right\rbrace \times \left[0.5,1\right]\leq \epsilon  \right) \right\rbrace 
\end{equation}
for some small $\epsilon >0$. Next, in the cases a) and b) we construct a fixed-point free expansive semi-algebraic sets for $f^p$ ($p$ large enough). Using infinitesimal perturbation if necessary we may assume that $f$ has only finitely many fixed points in $S_p$, say there are $m$.

Hence, let 
\begin{equation}
\mbox{Fix}\left( f\right) \cap S_p =\left\lbrace x_1,\cdots ,x_m\right\rbrace 
\end{equation}
where each fixed point is given by a pair of numbers, i.e. $x_i=\left(x_{i1},x_{i2}\right) ,1 \leq i\leq m$. Applying an additional perturbation we may assume that for all $1\leq i\leq m$ we have
\begin{equation}\label{firstco}
x_{i1}\in (0,1).
\end{equation}

\begin{remark}
A canonical way of doing this is adding an infinitesimal and working in a topos which is typical in synthetic differential geometry and then project back to the classical topos; however, it is clear that it would also be suffcient to consider small perturbations in a classical topos. 
\end{remark}

For each $i,~1\leq i\leq m$ consider the intervals
\begin{equation}
I_i:=\left\lbrace x_{i1}\right\rbrace \times \left[0,x_{i2}\right]  
\end{equation}
and
\begin{equation}
J_i:=\left\lbrace x_{i1}\right\rbrace \times \left[x_{i2},1\right]  
\end{equation}
Given the fixed point coordinates as in (\ref{firstco}) we may assume that $\epsilon$ is chosen small enough such that $I_i\subseteq A_p$ and $J_i\subseteq A_p$. Consider the interval $I_i\cup J_i$ in the set $\phi\left(\Pi_{4g}\right)$. It becomes the image of a loop homotop to the image of $a_1$ in $X$ under identification of sides. Note that the degree of the map induced by iterates $f^p$ on such an image of a loop increases as $p$ increases without limit. This implies that the map ${\mathbf f}^p_{\tiny \mbox{loc}}=\left(f^p_{\tiny \mbox{loc},1},f^p_{\tiny \mbox{loc},2} \right)$ is expansive on $S_p$ in cases a) and b), and expansive in one direction in conservative in the other direction in cases c) and d). Even more holds. In cases a) and b) given $1\leq i\leq m$ the map ${\mathbf f}^p_{\tiny \mbox{loc}}$ is even expansive on $S_p\setminus I_i$ or expansive on $S_p\setminus J_i$ (at least on one of the two subsets of $S_p$). Moreover for any $\epsilon>0$ we find a neighborhood of size $\epsilon_i$ of $I_i$, i.e. a set $B_{\epsilon}\left(I_i\right):=\left\lbrace x|\mbox{dist}_E\left(x, I_i\right)< \epsilon  \right\rbrace $ which contains a semia-lgebraic set $I^{\epsilon}_i\supset I_i$. Increasing $p$ if necessary we first have that he map ${\mathbf f}^p_{\tiny \mbox{loc}}$ is expansive on the semi-algebraic set $S_p\setminus I^{\epsilon}_i$. Then we get a set $M_m$ which is the union of $m$ semi-algebraic sets ($I^{\epsilon}_i$ or $J^{\epsilon}_i$) such that the map ${\mathbf f}^p_{\tiny \mbox{loc}}$ is expansive on $S_p\setminus M_m$. In cases c) and d) we find a set $L_m$ which is the union of $m$ semi-algebraic sets ($I^{\epsilon}_i$ or $J^{\epsilon}_i$) such that the map ${\mathbf f}^p_{\tiny \mbox{loc}}$ is expansive on $S_p\setminus L_m$ in one direction and conservative in the other direction. Hence in cases a) and b) we have constructed  a fixed-point free expansive semi-algebraic set $S_p\setminus M_m$ for $f^p$ ($p$ large enough) and in cases c) and d) we have constructed what may be called a fixed-point free hyperbolic semi-algebraic set $S_p\setminus M_m$ for $f^p$ ($p$ large enough).
This implies that the map ${\mathbf f}^p_{\tiny \mbox{loc}}=\left(f^p_{\tiny \mbox{loc},1},f^p_{\tiny \mbox{loc},2} \right)$ has a fixed point for any prime number beyond a certain threshold (according to our fixed point result in section 3). Since $S_p$ is constructed fixed-point free for any $p$ the prime number $p$ is the minimal period according to our considerations in the introduction.  
  
Having proved the theorem for polynomial mappings we finally extend to continuous functions completing the proof. Recall that we have constructed fixed points as intersections of algebraic curves which represent the image of loops on the surface such that the corresponding $1$-cohomology intersection number is $1$. Note that we have to establish the preservation of fixed points of certain iterates of functions on a fixed point free set, i.e. we have to ensure that in the limit the iterate $f^p$ has its fixed point on some fixed point free set ($\mbox{Fix}\left(f\right)=\oslash$). Consider again the localized map ${\mathbf f}^p_{\tiny \mbox{loc}}=\left(f^p_{\tiny \mbox{loc},1},f^p_{\tiny \mbox{loc},2} \right) :A_{\epsilon}\rightarrow {\mathbb R}^2$ on the semi-algebraic set $A_{p}\subset [0,1]\times [0,1]$ constructed as indicated above. Well, the sets $Z\left( f^p_{\tiny \mbox{loc},1}-\mbox{id}_1\right) $ and $Z\left( f^p_{\tiny \mbox{loc},2}-\mbox{id}_2\right)$ may look quite complicated. Especially, they may not contain algebraic curves connecting opposite faces of a cube of dimension $2$. Well, it may be possible to substitute the term 'algebraic curve' by something more general and try to extend intersection theory to this more generalized objects (but this is certainly not trivial, since the objects may be quite complicated from a topological point of view). However, there is a simpler way to get the fixed point property on a fixed point free set in the limit. For each $\epsilon >0$ we can construct a map ${\mathbf f}^p_{\epsilon, \tiny \mbox{loc}}=\left(f^p_{\epsilon ,\tiny \mbox{loc},1},f^p_{\epsilon ,\tiny \mbox{loc},2} \right) :A_{p}\rightarrow {\mathbb R}^2$ such that
\begin{equation}\label{appr}
|f^p_{\epsilon ,\tiny \mbox{loc},i}-f^p_{\tiny \mbox{loc},i}|\leq \epsilon
\end{equation}
for $i\in \left\lbrace 1,2 \right\rbrace$, and such that 
for some  $\delta >0$ we have
\begin{equation}
Z_{i,p,\delta}\subseteq Z\left( f^p_{\epsilon ,\tiny \mbox{loc},i}-\mbox{id}_i\right)
\end{equation}
where
\begin{equation}
Z_{i,p,\delta}= \bigcup\left\lbrace q_\delta \in Q_{\delta} |\exists x \in q_{\delta}: x\in 
Z\left( f^p_{\tiny \mbox{loc},i}-\mbox{id}_i\right) \right\rbrace . 
\end{equation}
Here, the $q_{\delta}$ are small $\delta\times \delta$-squares, which are elements of 
\begin{equation}
Q_{\delta}:=\left\lbrace q_{\delta}\subset \left[0,1\right]^2 |x\in q_{\delta}\Leftrightarrow \exists g_{\delta}\in G_{\delta}:~x=g_{\delta}+q\right\rbrace,
\end{equation}
where 
\begin{equation}
G_{\delta}:=\left\lbrace g_{\delta}=(g_{\delta 1},g_{\delta 2})|\exists~m,n\in {\mathbb N}: m\delta =g_{\delta 1}~\&~n\delta =g_{\delta 2}\right\rbrace ,
\end{equation}
\begin{equation}
q=\left\lbrace x=(x_1,x_2)|0\leq x_1\leq \delta~\&~0\leq x_2\leq \delta \right\rbrace ,
\end{equation}
and
\begin{equation}
g_{\delta}+q=\left\lbrace x| x=g_{\delta}+y~\mbox{ for some }~y\in q\right\rbrace . 
\end{equation}
Moreover, since the set $A$ is closed and fixed point free $\epsilon$ and $\delta$ above can be chosen small enough such that ${\mathbf f}^p_{\epsilon, \tiny \mbox{loc}}$ is fixed point free. For each $\epsilon >0$ and $\delta >0$ there is at least one 'fixed point square' of ${\mathbf f}^p_{\epsilon, \tiny \mbox{loc}}$, i.e. there is at least one $q_{\delta}\in Q_{\delta}$ such that
\begin{equation}\label{fixsquare}
q_{\delta}\cap_{i\in \left\lbrace 1,2\right\rbrace }Z\left( f^{p}_{\epsilon ,{\tiny \mbox{loc}},i}-\mbox{id}_i\right).
\end{equation}
Here we argue that there is a local polynomial mapping approximation such that (\ref{appr}) is satisfied for the polynomial mapping approximation, say for $\epsilon/2$. This polynomial mapping approximation of ${\mathbf f}^p_{\epsilon, \tiny \mbox{loc}}$ then may be assumed to have a fixed point constructed as the intersection of two real algebraic curves by the argument above. We may then construct the mapping ${\mathbf f}^p_{\epsilon, \tiny \mbox{loc}}$ from the latter polynomial mapping approximation such that (\ref{fixsquare}) is satisfied.
Hence, for a given prime number $p$ as $\epsilon\downarrow 0$ and $\delta \downarrow 0$ we get sequences $\epsilon_m,\delta_m$ converging to zero and a sequence of nested $\delta$-fixed point squares
\begin{equation}\label{nested}
q_{\delta_1}\supset q_{\delta_2} \supset q_{\delta_2}\cdots 
\end{equation}
where
\begin{equation}
q_{\delta_m}\cap_{i\in \left\lbrace 1,2\right\rbrace }Z\left( f^{p}_{\epsilon_m ,{\tiny \mbox{loc}},i}-\mbox{id}_i\right).
\end{equation}
As $\epsilon\downarrow 0$ we may construct ${\mathbf f}^p_{\epsilon, \tiny \mbox{loc}}$ uniformly converging to ${\mathbf f}^p$, i.e. with respect to the supremum norm. Then
\begin{equation}
x_f\in \bigcap_{m\in {\mathbb N}} q_{\delta_m}\neq \oslash
\end{equation}
(by completeness of ${\mathbb R}$) is a fixed point of $f^p$ and a periodic point of minimal period $p$ of $f$ (because $A$ is fixed point free). Since the prime number was arbitrary (except being required to be large enough, i.e. to exceed some threshold $n_0\in {\mathbb N}$), we arrive at the conclusion.
\section{Further remarks}

We have established a necessary and sufficient algebraic criterion which tells us whether a given homotopy class has infinitely many periodic points. The investigation may be refined in order to determine all mimimal periods of a given homotopy class of a functon $f:X\rightarrow X$, i.e. in order to determine the set $\mbox{MPer}\left(f\right)$. A second goal may be to determine the first homology morphisms which satisfy the algebraic condition of theorem 1 explicitly. This may wor for the $2$-torus and the bretzel surface for obvious reasons. However, we emphasize that the algebraic criterion of theorem 1 is an efficiently decidable criterion, i.e. as long we can compute the first order homology morphism of a given function $f:X\rightarrow X$ we may decide whether its homotopy class has infinitely many periodic points by computation of certain approximations of the zeros' of its characteristic polynomial. Well this can be done in an efficient way. Furthermore, let us note that the approach of fixed point theory of iterates from the point of view of real algebraic geometry is not limited to surfaces and orientable surfaces in particular but may be extended to higher-dimensional objects and non-orientable surfaces.

\end{document}